\documentclass{amsart}

\usepackage{mathrsfs}
\usepackage{amsmath}
\usepackage{amssymb}
\usepackage{color}

\newtheorem{theorem}{Theorem}[section]
\newtheorem{corollary}[theorem]{Corollary}
\newtheorem{lemma}[theorem]{Lemma}
\newtheorem{proposition}[theorem]{Proposition}
\theoremstyle{definition}
\newtheorem{definition}[theorem]{Definition}
\newtheorem{remark}[theorem]{Remark}

\numberwithin{equation}{section}

\newenvironment{pot2}{\paragraph{Proof of Theorem \ref{INV-Conv}}}{\hfill$\qed$}
\newenvironment{pot3}{\paragraph{Proof of Theorem \ref{Stoke}}}{\hfill$\qed$}
\newenvironment{pot4}{\paragraph{Proof of Theorem \ref{MCIEq88}}}{\hfill$\qed$}

\numberwithin{equation}{section}

\begin{document}
\title{Scaling limits of discrete holomorphic functions}
\author[G. B. Ren]{Guangbin Ren}
\author[Z. P. Zhu]{Zeping Zhu}
\address{Guangbin Ren, Department of Mathematics, University of Science and
Technology of China, Hefei 230026, China}
\email{rengb$\symbol{64}$ustc.edu.cn}
\address{Zeping Zhu, Department of Mathematics, University of Science and
Technology of China, Hefei 230026,
China}
\email{zzp$\symbol{64}$mail.ustc.edu.cn}

\thanks{This work was supported by the NNSF  of China (11071230).}

\begin{abstract}  One of the most natural and
challenging  issues in discrete complex analysis is to
prove the
convergence of discrete holomorphic functions to their continuous counterparts.
 This article is to solve the open problem
 in the general setting. To this end we introduce
 new concepts of {\it discrete  surface measure} and {\it discrete  outer normal vector} and establish the discrete  Cauchy-Pompeiu integral formula,
 \begin{eqnarray*} f(\zeta)=\displaystyle{\int_{\partial B^h}} \mathcal{K}^h(z,\zeta) f(z)dS^h(z)+\displaystyle{\int_{ B^h}} E^h(\zeta-z) \partial_{\bar z}^h f (z)dV^h(z),\end{eqnarray*}
which results in   the uniform convergence of the scaling limits of discrete holomorphic functions up to  second order derivatives in the standard square lattices.
\end{abstract}

\keywords{Discrete complex analysis, Discrete holomorphic function, Green's theorem, Cauchy-Pompeiu formula, Scaling limit.}
\subjclass[2010]{30G25, 39A12, 49M25, 31C20}

\maketitle
\baselineskip20pt
\section{Introduction}

Discrete complex analysis aims to find a kind of mathematical theory on lattices similar to its continuous
counterpart.
A typical problem in  discrete complex analysis is  the  convergence of the scaling limits of discrete holomorphic functions.

A convergence problem for a certain kind of discrete holomorphic functions has been
studied  by Smirnov and his collaborators  and   it is of eminent importance in proving the  conjecture about the conformal invariance in the Ising model  \cite{36,37,38}.
 Skopenkov \cite{35}   considered a different   convergence problem  and he  proved that the Dirichlet boundary value problem for the real part of a discrete analytic function has a unique solution and this solution uniformly converges to a harmonic function in the case of orthogonal lattices. This issue  was also investigated  by
 Courant-Friedrichs-Lewy \cite{31} for square lattices, and by Chelkak-Smirnov \cite{33} for rhombic lattices.
Nevertheless, the
convergence of the scaling limits of discrete holomorphic functions  in the general case is still an open problem.

The goal of this article is to solve this  open  problem  in the case of the standard square lattices.
 This will depend heavily  on  our   discrete Cauchy-Pompeiu integral formula
\begin{eqnarray*} f(\zeta)=\displaystyle{\int_{\partial B^h}} \mathcal{K}^h(z,\zeta) f(z)dS^h(z)+\displaystyle{\int_{ B^h}} E^h(\zeta-z) \partial_{\bar z}^h f (z)dV^h(z),\end{eqnarray*}
which clearly respects its continuous version.

To this end, a new version of the integral theory in the  discrete complex analysis is developed on square lattices in the article.
The classical one was initiated by  Isaacs \cite{1} and  Ferrand  \cite{2},
and further developed  by Duffin \cite{3},  Zeilberger \cite{8}, and others.
This function theory has been generalized to the cases of more complicated graphs by Mercat  \cite{17} and Bobenko-Mercat-Suris \cite{18}. Moreover,
discrete complex analysis has found its applications in fields ranging from
 combination geometry \cite{26,Steph}, numerical analysis \cite{23}, computer graphics \cite{24,25}, and statistical physics \cite{21}.
 However, the theory does not fully mature  even on square lattices.

Our new theory has several  advantages over the  old ones. Firstly,   it distinguishes itself by the elegant discrete Cauchy-Pompeiu integral formula,
which is totally  analogous to its continuous version.  Therefore it has good performance  in applications.
Next,   as opposed to discrete Clifford theory  where maps from $\mathbb Z_h^2$ to $\mathbb R^{16}$ have to be  considered \cite{Somm}, we can  consider maps from $\mathbb Z_h^2$ to $\mathbb C$ as expected without raising the dimension of the target domain.
Finally,  we can
prove the uniform convergence  of discrete holomorphic function up to second order  derivatives.

It is worth pointing  out that
although our  results are stated  only in the complex plane,  our approach is    applicable   to   higher dimensional spaces.

This paper is organized as follows:
In Section \ref{S2}, we introduce some basic concepts  such as {\it discrete  surface measure} and {\it discrete  outer normal vector}. In Sections \ref{S3} and \ref{S5} we establish the integral theory of discrete holomorphic functions; in particular, we obtain the discrete Cauchy-Pompeiu formula. In Section \ref{S6} and \ref{S7}, we study the convergence problem of discrete holomorphic functions and solve the open problem as stated in  Theorems \ref{Conv}, \ref{Conv2}, \ref{INV-Conv} and \ref{Conv-D}.
Appendix  in Section \ref{S9} contains all the technical results.

\section{Discretization of operators, surface measure, and normal vectors  }\label{S2}

Basic  elements  in discrete complex analysis are introduced in this section.

\subsection{Discrete $\bar\partial$-operator}
 We will work on the discrete lattices
 $\mathbb{Z}_h^2$, a discretization of $\mathbb{R}^2$.
Here  $\mathbb{Z}_h=h\mathbb Z$  for any given positive parameter $h$.
We shall study discrete  holomorphic functions related to the symmetric discretization of the classical $\bar\partial$-operator.

\begin{definition}  The discrete $\bar\partial$-operator and its conjugate are defined respectively  as
 $$\partial_{\bar z}^h:=\frac{1}{2}(\partial_1^h+i\partial_2^h),
 \qquad \partial_{ z}^h:=\frac{1}{2}(\partial_1^h-i\partial_2^h),$$
 where $\partial_i^h$ $(i=1,2)$ are symmetric difference operator
 $$\partial_i^h=\displaystyle{\frac{1}{2}}(\partial_i^{+,h}+\partial_i^{-,h}). \qquad $$
Here $\partial_i^{+,h}$ and $\partial_i^{-,h}$ stand for
 the  forward and backward difference operators respectively, i.e.,
  \begin{eqnarray*}
  \partial_i^{+,h}f(x)&=&\frac{f(x+he_i)-f(x)}{h},\\
  \partial_i^{-,h}f(x)&=&\frac{f(x)-f(x-he_i)}{h},
  \end{eqnarray*}
  where $\{e_1, e_2\}$ is  the standard basis of $\mathbb R^2$.
\end{definition}

 Notice that  the symmetric discretization $\partial_{\bar z}^h$  converges to the classical differential operator $\partial_{\bar z}$ as $h$ tends to zero.

It is worth noting that $\partial_{\bar z}^hf $ makes sense on $B$ only for those functions $f$ whose
  definition domain contain $\overline{B}$, the discrete closure of $B$.

 \begin{definition}\label{def:discrete-boudary}
 Let $B$ be a subset of $\mathbb{Z}_h^2$. We define
its discrete closure and  interior respectively as
$$\overline{B}:=B\cup\partial B, \qquad B^\circ:=B\setminus \partial B, $$
 where $\partial B$ is the \textit{discrete boundary} of $B$ consisting of every point $z \in \mathbb{Z}_h^2$ whose
 neighborhood
 $$
  N(z) := \{ z, z\pm h, z\pm hi \}
  $$
  has  some point inside $B$ and some other point outside $B$, i.e.,
 $$
 \partial B := \Big\{z \in \mathbb{Z}_h^2: N(z) \cap B \neq \emptyset \quad \mathrm{and}\quad N(z) \setminus B\neq \emptyset \Big\}.
 $$
 \end{definition}

Now we can introduce the concept of discrete holomorphic functions.

\begin{definition}\label{def:dis-hol}
 A function $f:\overline{B} \longrightarrow \mathbb{C}$ with  $B\subset \mathbb{Z}_h^2$ is said to be discrete holomorphic on $B$ if for any $z \in B$ we have $\partial_{\bar z}^hf(z)=0. $
\end{definition}

\subsection{Discrete   surface measure and discrete  normal vector}\label{S3}
The  concepts of   discrete   surface measure and discrete outer normal vector
are essential to our theory.

\begin{definition}\label{EDF}
Let $B$ be a subset of $\mathbb{Z}_h^2$. The discrete boundary measure $S$ on $\partial B$  is defined  as
$$S(U)=\sum_{z\in U}s(z),\quad \forall\text{ } U\subset\partial B, $$
where  $s:\partial B \longrightarrow \mathbb{R}$ is the density function
\begin{eqnarray}\label{def:density-function-s}s=\frac{h^2}{2}\sqrt{\sum_{i=1}^2\Big((\partial_i^{+,h}\chi_B)^{ 2}+(\partial_i^{-,h}\chi_B)^{ 2}}\Big),
\end{eqnarray}
and $\chi_B$ denotes the characteristic function of  $B$.
\end{definition}

\begin{remark} If we extend  the density function $s$ to the whole $\mathbb{Z}_h^2$
by  zero extension, then  (\ref{def:density-function-s})  holds true on  $\mathbb Z_h^2$ since $$\sum_{i=1,2}\Big((\partial_i^{+,h}\chi_B)^{ 2}+(\partial_i^{-,h}\chi_B)^{ 2}\Big)=0 \qquad \mbox{on}\quad  \mathbb{Z}_h^2\backslash\partial B.$$
\end{remark}

\begin{definition}\label{ENV}
 The discrete outer normal vector at a boundary point of $B\subset \mathbb{Z}_h^2$
 is a vector
$$\vec{n}=(n_1^+,n_1^-,n_2^+,n_2^-),$$
defined by
$$n_l^{\pm}=\frac{-2\;\partial_l^{\pm,h}\chi_B}{\sqrt{\sum\limits_{i=1}^2\left((\partial_i^{+,h}\chi_B)^{ 2}+(\partial_i^{-,h}\chi_B)^{ 2}
\right)}}, \qquad l=1,2.$$
\end{definition}
It is evident that the Euclidean norm of $\vec{n}$ is always equal to $2$ on $\partial B$.

\section{Green's formula on $\mathbb{Z}_h^2$}\label{S3}

With the concepts of  the discrete surface measure $S$ and the  discrete outer normal vector $\vec{n}$ above, we can now establish the discrete version of  Green's formula  in this section.

Let  $V^h$ be the Haar measure on the group $\mathbb{Z}_h^2$. Then for any $f:\mathbb{Z}_h^2\to \mathbb{C}$ we have
$$\int_{\mathbb{Z}_h^2}fdV^h=\sum_{z\in \mathbb{Z}_h^2}f(z)h^{2}.$$
As usual, for any  $B\subset \mathbb{Z}_h^2$,
$$\int_{B}fdV^h:=\int_{\mathbb{Z}_h^2}f\chi_BdV^h.$$

\begin{theorem}[Green's formula]\label{Stoke}
Let $B$ be a bounded subset of $\mathbb{Z}_h^2$. For any function $f:\overline{B}\longrightarrow\mathbb{R}$, we have
$$\int_{\partial B}fn_i^{\pm}dS=\int_B\partial_i^{\mp,h}fdV^h, \quad (i=1,2). $$
\end{theorem}

In the discrete setting, we shall always extend $S$, $s$ , and $n_i^{\pm}$ to the whole lattice  $\mathbb{Z}_h^2$. That is, we identify  the discrete boundary measure $S$   with $S\circ\varsigma^{-1}$ via the natural embedding map $\varsigma:\partial B \longrightarrow \mathbb{Z}_h^2$
and identify the density function $s$ and the discrete normal vectors $n_i^{\pm}$ with their zero extensions respectively.
Notice that the density function of   $S\circ\varsigma^{-1}$ is exactly the zero extension of $s$.

We shall state our results in the language of distributions. To this end, we define the space of discrete test functions by $$\mathcal D(\mathbb{Z}_h^2):=\{f:\mathbb{Z}_h^2\longrightarrow\mathbb{R}|\textbf{supp}f\text{ is bounded}\}$$ and the discrete distribution space by $$\mathcal D^{*}(\mathbb{Z}_h^2):=\textrm{Hom}(\mathcal  D(\mathbb{Z}_h^2),\mathbb{R}).$$
It is well-known that $\mathcal  D^{*}(\mathbb{Z}_h^2)$ contains every function $f:\mathbb{Z}_h^2\longrightarrow \mathbb{R}$ regarded as a discrete distribution
$$\Lambda_f(g):=\int_{\mathbb{Z}_h^2}fgdV^h, \qquad \forall\ g \in \mathcal  D(\mathbb{Z}_h^2). $$
Furthermore, it also contains every locally finite
 measure $\mu$ on the topological group $\mathbb{Z}_h^2$  since it can be identified with a discrete distribution
$$\Lambda_{\mu}(g):=\int_{\mathbb{Z}_h^2}gd\mu, \qquad \forall\ g \in \mathcal  D(\mathbb{Z}_h^2). $$

\begin{lemma}[Discrete Stokes equations]\label{lem: StokeEq1}
Let $B$ be a subset of $\mathbb{Z}_h^2$, $S$  the surface measure on $\partial B$, and $\vec{n}=(n_1^+,n_1^-,  n_2^+, n_2^-)$  the normal vector on $\partial B$.
Then in the sense of distribution we have
\begin{eqnarray}\label{StokeEq1}
   -\partial_i^{\pm,h}\chi_B&=&n_i^{\pm}S, \qquad (i=1,2)\\
  4\chi_{\partial B}&=&\sum\limits_{i=1,2}(n_i^+)^2 +  (n_i^-)^2. \label{StokeEq2}
\end{eqnarray}
\end{lemma}
\begin{proof}
First we check the second identity.
By definition, we have $$\left.\sum\limits_{i=1,2}(n_i^+)^2 +  (n_i^-)^2 \right|_{\partial B}\equiv 4.$$ Since $n_i^{\pm}(i=1,2)$ vanish on $\mathbb{Z}_h^2\backslash B$, we have $$\left.\sum\limits_{i=1,2}(n_i^+)^2 +  (n_i^-)^2\right|_{\mathbb{Z}_h^2\backslash B} \equiv 0.$$
Thus (\ref{StokeEq2}) holds true.

Next we check identity (\ref{StokeEq1}). In $\mathcal D^{*}(\mathbb{Z}_h^2)$, we have  \begin{equation}\label{Eq:S-s}S=h^{-2}s,\end{equation} since $s$ is the density function of $S$. Indeed, for any $f\in D(\mathbb(Z)_h^2)$ we have
\begin{gather*}\begin{split}
         \langle S,f\rangle=&\int_{\mathbb{Z}_h^2}fd(S\circ\varsigma^{-1})
              =\int_{\partial B}fdS
              =\sum_{z \in \partial B}f(z)s(z)\\
              =&\sum_{z \in \mathbb{Z}_h^2}f(z)s(z)\qquad \text{(zero extension)}\\
              =&\int_{\mathbb{Z}_h^2}h^{-2}sfdV^h\\
              =&\langle h^{-2}s,f\rangle
\end{split}
\end{gather*}
as desired.
So we need to show that
$$-\partial_i^{\pm,h}\chi_B=h^{-2}n_i^{\pm}s. $$

For any $x\in \partial B$,  it follows from Definitions \ref{EDF} and \ref{ENV} that
$$\begin{array}{lll}n_i^{\pm}(x)s(x)&=&
                                \displaystyle{\frac{-2\partial_i^{\pm,h}\chi_B}{\sqrt{\sum\limits_{i=1,2}(\partial_i^{+,h}\chi_B)^{ 2}+(\partial_i^{-,h}\chi_B)^{ 2}}}\frac{h^2}{2}\sqrt{\sum_{i=1,2}(\partial_i^{+,h}\chi_B)^{ 2}+(\partial_i^{-,h}\chi_B)^{ 2}}}\\
                                &=&
                                -h^2\partial_i^{\pm,h}\chi_B. \end{array}$$
On the other hand, for any $x \notin \partial B$ we have $$s(x)=\partial_i^{\pm,h}\chi_B(x)=0.$$
Hence $$-\partial_i^{\pm,h}\chi_B=h^{-2}n_i^{\pm}s. $$ This completes the proof.
\end{proof}

\begin{remark}
In the continuous setting the Green theorem $$\displaystyle{\int_{\partial B}fn_idS =\int_B \frac{\partial f}{\partial x_i}dV}$$ can be restated  in term of distribution  as
\begin{eqnarray}\label{dis-StokeEq1} n_iS=-\frac{\partial}{\partial x_i}\chi_B, \quad (i=1,2).\end{eqnarray}
Here $n_i$ is identified with  its zero extension from $\partial B$ to $\mathbb{R}^2$ and the surface measure $S$ is identified with its push-out $S\circ\varsigma^{-1}$ via the classical  embedding map $\varsigma$ from $\partial B$ to $\mathbb{R}^2$.
Moreover, in this point of view we have
\begin{eqnarray}\label{dis-StokeEq2}\sum_{i=1}^{2}n_i^2=\chi_{\partial B}.
\end{eqnarray}

The system of discrete Stokes equations in Lemma \ref{StokeEq1} is a  variant of  its continuous counterparts in (\ref{dis-StokeEq1}) and  (\ref{dis-StokeEq2}).
However,  the discrete normal vectors is likely to diverge under scaling limits.

We point out that the discrete surface measure $S$ and the discrete normal vector $\vec{n}$ are uniquely determined by the system of Stokes equations in Lemma \ref{lem: StokeEq1}. Its proof is the same as  in the continuous version as we shall see in  the next lemma.
\end{remark}

\begin{lemma}\label{lem: StokeEq4}
Let $B$ be a  domain in $\mathbb{C}$ with smooth boundary $\partial B$.
   Then the system of the Stokes equations
\begin{eqnarray}\label{StokeEq8}
  m_iT&=&-\frac{\partial}{\partial x_i}\chi_B, \quad (i=1,2)\\
 \sum_{i=1}^{2}m_i^2&=&\chi_{\partial B}, \label{StokeEq9}
\end{eqnarray}
with  $T$ being a   non-negative regular Borel measure  on $\partial B$, and $\vec{m}=(m_1,  m_2)$ smooth on $\partial B$, has  a unique  solution $\{T, \vec{m}\}=\{S, \vec{n}\}$,
where $S$   is the boundary  measure and $\vec{n}$   the outer normal vector on $\partial B$.
 \end{lemma}

\begin{proof}
We only need to prove the uniqueness since $\{S, \vec{n}\}$ is clearly a solution.

Assume $\{T,\vec{m}\}$ is another solution.
For any non-negative $f \in C_0^{\infty}(\mathbb{C})$,
$$\int_{\partial B} fdT=\int_{\partial B} f m_1^2dT+\int_{\partial B} f m_2^2dT.$$

Since $\{S,\vec{n}\}$ is also a solution,  then we have $n_iS=m_iT\quad(i=1,2)$. It follows that
\begin{eqnarray*}\int_{\partial B} fdT&=&\int_{\partial B} f m_1^2dT+\int_{\partial B}f m_2^2dT\\
                                    &=&\int_{\partial B} f m_1n_1dS+\int_{\partial B} f m_2n_2dS,
\end{eqnarray*}
which implies that $$\int_{\partial B} fdT\leqslant \int_{\partial B} fdS$$ since $\vec{m}$ and $\vec{n}$ are both unit vector and $f$ is non-negative.

By symmetricity, we have    $$\int_{\partial B} fdS\leqslant \int_{\partial B} fdT.$$
This implies  $T=S$ since the set of finite combinations of non-negative smooth functions with compact support is dense in $C_c(\partial B)$.

Since $n_iS=m_iT\quad(i=1,2)$ and $S=T$, we have $n_i=m_i$ a.e.$S$ on $\partial B$. According to the assumption that $\vec{m}$ and $\vec{n}$ are both smooth, we obtain that $$\vec{m}\equiv \vec{n}. $$
\end{proof}

Now we are in position to   prove our main result in this section.

\begin{pot3}
Put $$g=f\chi_{\overline{B}}. $$
Since $\overline{B}$ is bounded,    we have $g \in \mathcal{D}(\mathbb{Z}_h^2)$
so that
 (\ref{StokeEq1}) implies
\begin{eqnarray*}
\langle n_i^{\pm}S,g\rangle&=&\langle-\partial_i^{\pm,h}\chi_B, g\rangle\\
              &=&-\int_{\mathbb{Z}_h^2}\partial_i^{\pm,h}\chi_B(x)g(x)dV^h(x).
\end{eqnarray*}
It is easy to verify that $$-\int_{\mathbb{Z}_h^2}\partial_i^{\pm,h}\chi_B(x)g(x)dV^h(x)
=\int_{\mathbb{Z}_h^2}\chi_B(x)\partial_i^{\mp,h}g(x)dV^h(x).$$
We thus  have
 \begin{equation}\label{Pr.Eq}
              \langle n_i^{\pm}S,g\rangle=\int_{B}\partial_i^{\mp,h}gdV^h=\int_{B}\partial_i^{\mp,h}fdV^h.
 \end{equation}

On the other hand, it follows from identity (\ref{Eq:S-s})  that
\begin{eqnarray*}
\langle n_i^{\pm}S,g\rangle&=&\langle h^{-2}n_i^{\pm}s, g\rangle=\int_{\mathbb{Z}_h^2}h^{-2}n_i^{\pm}sgdV^h=\sum_{z\in \mathbb{Z}_h^2}n_i^{\pm}(z)s(z)g(z)\\
&=&\sum_{z\in \partial B}n_i^{\pm}(z)s(z)g(z)\qquad \text{(zero extension)}\\
&=&\int_{\partial B}gn_i^{\pm}dS.
\end{eqnarray*}
Since $f=g$ on $\partial B$, we thus obtain
\begin{eqnarray*}
\langle n_i^{\pm}S,g\rangle
              &=&\int_{\partial B}fn_i^{\pm}dS,
\end{eqnarray*}
This together with  (\ref{Pr.Eq})  leads to the Green theorem.
\end{pot3}

\section{Discrete  Cauchy-Pompeiu integral formula}\label{S5}

In this section we introduce the  discrete Bochner-Matinelli   kernel  and establish the discrete  Cauchy-Pompeiu integral formula.

The fundamental solution \cite{R}
of operator $\partial_{\bar z}^h$ is defined by $$E^h(x,y)=\frac{1}{h}E\Big(\frac{x}{h},\frac{y}{h}\Big), $$
where
\begin{equation}\label{Fund-Sol}E(x,y)=\frac{1}{4\pi^2}\int_{[-\pi,\pi]^2}\frac{2}{i\sin{u}-\sin{v}}e^{i(ux+vy)}dudv. \end{equation}
Let $\delta_0^h$ be the discrete Dirac delta function on $\mathbb{Z}_h^2$, i.e.,
\begin{equation*}
\delta_0^h(z)=\left\{\begin{array}{lll}
                        h^{-2},&&\quad z=0,\\
                        0,&&\quad z \neq 0.
                       \end{array}
                 \right.
\end{equation*}

\begin{theorem}\cite{R}Function $E^h$ is the fundamental solution of $\partial_{\bar z}^h$, i.e.,
$$\partial_{\bar z}^hE^h=\delta_{0}^h \qquad\text{in}\ \ \mathbb{Z}_h^2. $$
\end{theorem}
\begin{proof}
Without loss of generality, we can assume $h=1$.
By direct calculation,
$$\partial_{\bar z}^h(e^{i(ux+vy)})=\frac{i\sin u-\sin v}{2}e^{i(ux+vy)}.$$
Applying $\partial_{\bar z}^h$ on both sides of (\ref{Fund-Sol}), we thus obtain
$$\partial_{\bar z}^hE^h=\frac{1}{4\pi^2}\int_{[-\pi,\pi]^2}e^{i(ux+vy)}dudv=\delta_0^h.$$
\end{proof}

Let $\mathcal{K}^h$  be  the discrete Bochner-Matinelli   kernel, defined by
 $$\mathcal{K}^h(z,\zeta)= A(z-\zeta)n_1^{-}(z)+ B(z-\zeta) n_1^{+}(z)+ iC(z-\zeta)n_2^{-}(z)+iD(z-\zeta)n_2^{+}(z),$$
where
\begin{eqnarray*}
A(z)&=&
-4^{-1}E^h(h-z),
\\
B(z)&=&-4^{-1}E^h(-h-z),
\\
C(z)&=&-4^{-1}E^h(ih-z),
\\
D(z)&=&-4^{-1}E^h(-ih-z).
\end{eqnarray*}

\begin{theorem}[Cauchy-Pompeiu]\label{MCIEq11}Let $B$ be a bounded subset of $\mathbb{Z}_h^2$. Then for any function $f: B\longrightarrow \mathbb{C}$ we have
\begin{eqnarray}\label{eq:BM}\chi_B(\zeta) f(\zeta)=\displaystyle{\int_{\partial B}} \mathcal{K}^h(z,\zeta) f(z)dS(z)+\displaystyle{\int_{ B}} E^h(\zeta-z) \partial_{\bar z}^h f (z)dV^h(z).\end{eqnarray}
\end{theorem}

\begin{proof} We first split the first summand in the right side of (\ref{eq:BM}) into four parts:
\begin{gather}\label{EqMCI}\begin{split}
 \displaystyle{\int_{\partial B}}\mathcal{K}^h(z,\zeta)f(z) dS(z) =&\displaystyle{\int_{\partial B}} A(z-\zeta) f(z)n_1^{-}(z)dS(z)+\\
                                          &\displaystyle{\int_{\partial B}} B(z-\zeta)  f(z)n_1^{+}(z)dS(z)+\\
                                          &i\displaystyle{\int_{\partial B}} C(z-\zeta)  f(z)n_2^{-}(z)dS(z)+\\
                                          &i\displaystyle{\int_{\partial B}} D(z-\zeta)  f(z)n_2^{+}(z)dS(z)\\
                                          =&I_1+I_2+I_3+I_4.
\end{split}
\end{gather}

By definition, we have
 $$A=-4^{-1}\tau_x^{-1}\rho E^h, \quad B=-4^{-1}\tau_x\rho E^h, \quad C=-4^{-1}\tau_y^{-1}\rho E^h, \quad D=-4^{-1}\tau_y\rho E^h,$$
where the operator $\rho$ is the reflection
$$\rho f(z)=f(-z),$$ and  $\tau_x, \tau_y$ are translations  $$\tau_xf(z)=f(z+h), \qquad \tau_yf(z)=f(z+hi).$$

Applying Theorem \ref{Stoke} to $I_1$, we have
\begin{equation*}
\begin{split}
I_1=&\displaystyle{\int_{ B}}\partial_1^{+,h}[ A(\cdot-\zeta) f(\cdot)](z)dV(z)\\
    =&\displaystyle{\int_{ B}} \tau_{x}A(z-\zeta)\partial_1^{+,h}f(z) +\partial_1^{+,h}A(z-\zeta)f(z) dV^h(z).
\end{split}
\end{equation*}
Similarly, we have
\begin{gather*}
 \begin{split}
 I_2
    =&\displaystyle{\int_{ B}} \tau_{x}^{-1}B(z-\zeta)\partial_1^{-,h}f(z) +\partial_1^{-,h}B(z-\zeta)f(z) dV^h(z),\\
 I_3
    =&i\displaystyle{\int_{ B}} \tau_{y}C(z-\zeta)\partial_2^{+,h}f(z)+ \partial_2^{+,h}B(z-\zeta) f(z) dV^h(z),\\
 I_4
    =&i\displaystyle{\int_{ B}}  \tau_{y}^{-1}D(z-\zeta) \partial_2^{-,h}f(z) + \partial_2^{-,h}C(z-\zeta)f(z) dV^h(z).
 \end{split}
\end{gather*}
Substituting the identities above to  (\ref{EqMCI}), we obtain
\begin{gather*}
 \begin{split}
  &\displaystyle{\int_{\partial B}} \mathcal{K}^h(z,\zeta) f(z) dS(z)\\
 =&\displaystyle{\int_{ B}}\big(\tau_{x}A(z-\zeta)\partial_1^{+,h}f(z)
  + \tau_{x}^{-1}B(z-\zeta)\partial_1^{-,h}f(z)+i\tau_{y}C(z-\zeta) \partial_2^{+,h}f(z)\\
  &  +i\tau_{y}^{-1}D(z-\zeta) \partial_2^{-,h}f(z)\big) dV^h(z)+\displaystyle{\int_{ B}}
                                                            \big(\partial_1^{+,h}A(z-\zeta)+ \partial_1^{-,h}B(z-\zeta)\\
  &+ i\partial_2^{+,h}C(z-\zeta) + i\partial_2^{-,h}D(z-\zeta) \big) f(z)dV^h(z).
 \end{split}
\end{gather*}

By direct calculation. we have $$\tau_{x}A  =\tau_x^{-1}B=\tau_{y}C =\tau_{y}^{-1}D=-4^{-1}\rho E$$ and  $$\partial_1^{+,h}A+ \partial_1^{-,h}B+i\partial_2^{+,h}C+ i\partial_2^{-,h}D =\rho (\partial_{\bar z}^hE^h)=\delta_0^h. $$
These  lead to $$\displaystyle{\int_{\partial B}} \mathcal{K}^h(z,\zeta) f(z)dS(z)=-\displaystyle{\int_{ B}} E^h(\zeta-z) \partial_{\bar z}^h f (z)dV^h(z)+\displaystyle{\int_{B}} \delta_0^h(z-\zeta) f(z)dV^h(z) $$ as desired.
\end{proof}

\begin{corollary}\label{MCIEq} If $f$ is discrete  holomorphic on  a bounded subset $B$ of $\mathbb{Z}_h^2$, then
\begin{equation*}
  \chi_B(\zeta)f(\zeta)=\int_{\partial B}\mathcal{K}^h(z,\zeta)f(z)dS(z).
\end{equation*}
\end{corollary}

\begin{remark}
The preceding theorem indicates  a new phenomena that discrete holomorphic functions behave  differently with continuous counterparts  on boundaries since for any discrete holomorphic function
we have
$$
 \int_{\partial B}\mathcal{K}^h(z,\zeta)f(z)dS(z)=\left\{\begin{array}{lll}
                        f(\zeta),&&  \zeta\in \partial^{+} B;\\
                        0,&&  \zeta  \in \partial^{-} B.
                       \end{array}
                 \right.
$$
where $\partial^{\pm} B$ constitute a partition of $\partial B$, defined by
\begin{eqnarray*}
\partial^{+} B&:=&\partial B \cap B,
\\
 \partial^{-} B&:=&\partial B \setminus B.
\end{eqnarray*}
\end{remark}

Finally, we study the holomorphicity of the Bochner-Matinelli kernel. It turns out that  $\mathcal{K}^h(z, \cdot)$ is discrete holomorphic outside the neighbourhood of the diagonal.

\begin{theorem}\label{MCIEq88} For any given $z\in\partial B$,
the discrete Bochner-Matinelli kernel $\mathcal{K}^h(z, \cdot)$ is discrete holomorphic on
$(\mathbb Z_h^2\setminus \partial B)\bigcup (\mathbb Z_h^2\setminus   N(z))$.
\end{theorem}

We leave the proof to  appendix I since its proof is direct but unpleasant.

\begin{remark} To consider the  holomorphicity of the kernel $\mathcal{K}^h$ along the neighborhood of the diagonal, we denote
 $$\Gamma:=\big\{(z, \zeta): z\in \partial^- B, \zeta\in B\cap N(z)\}\bigcup
\big\{(z, \zeta): z\in \partial^+ B, \zeta\in N(z)\setminus B\}.$$
Then one can verify from the proof of Theorem \ref{MCIEq88} that
$$\partial_{\bar \zeta}^h\mathcal{K}^h(z,\zeta)=0, \qquad (z, \zeta)\not\in\Gamma. $$
and when $(z, \zeta)\in\Gamma$
we have $\zeta\in N(z)$ and
\begin{eqnarray*}
\partial_{\bar \zeta}^h\mathcal{K}^h(z,\zeta)=
\begin{cases}  -\frac{1}{4h^2}n_1^{+}(z),& \qquad \zeta=z+h,\\
                -\frac{1}{4h^2}n_1^{-}(z),& \qquad \zeta=z-h,\\
                -\frac{1}{4h^2}n_2^{+}(z),& \qquad \zeta=z+hi,\\
                -\frac{1}{4h^2}n_2^{-}(z),& \qquad \zeta=z-hi. \\
\end{cases}
\end{eqnarray*}
\end{remark}

\section{Approximation and convergence}\label{S6}

For the convergence and approximation, we shall see that a function  is holomorphic if and only if it is the scaling limit of discrete holomorphic functions.

\subsection{Approximation}

For  any given  holomorphic function $f$, we come to
 construct  discrete holomorphic functions $f^{h}$ converging to  $f$.

First, we need a  concept about the convergence of discrete sets.

\begin{definition}\label{CDC}
Let $B$ be a bounded open set in $\mathbb{C}$. We say $B^h\subseteq \mathbb{Z}_h^2$ converges to $B$ and denote as
$$\lim_{h\to 0^+} B^h=B$$
if  the distances
between $\partial B$ and $\partial B^h$
 as well as between $\overline B$ and $B^h$  converge  to zero, i.e.,
$$\begin{array}{ll}
\displaystyle{\lim_{h\rightarrow 0^+}}&\displaystyle{\max_{\alpha \in\partial B}\min_{\beta\in\partial B^h}||\alpha-\beta||=0,}\\
\displaystyle{\lim_{h\rightarrow 0^+}}&\displaystyle{\max_{\alpha \in\partial B^h}\min_{\beta\in\partial B}||\alpha-\beta||=0,}\\
\displaystyle{\lim_{h\rightarrow 0^+}}&\displaystyle{\max_{\alpha \in \overline{B}}\min_{\beta\in B^h}||\alpha-\beta||=0,}\\
\displaystyle{\lim_{h\rightarrow 0^+}}&\displaystyle{\max_{\alpha \in B^h}\min_{\beta\in \overline{B}}||\alpha-\beta||=0.}
\end{array}$$
\end{definition}

Recall  for any $\Omega \subset \mathbb{Z}_h^2$ the interior of $\Omega$ is defined by $$\Omega^\circ=\Omega\setminus \partial \Omega,$$
and $S^h$ and $V^h$  represent the discrete surface  measure on $B$ and the Haar measure in the lattice $\mathbb{Z}_h^2$, respectively.
Let $H(B)$ denote the space of holomorphic functions in the domain $B\subset\mathbb C$.

Now we come to the first main result in this subsection.

\begin{theorem}\label{Conv}
Let $B$ be a bounded open set in $\mathbb{C}$ and set $$B^h:=(B\bigcap\mathbb{Z}_h^2)^{\circ}\subset \mathbb{Z}_h^2.$$
If $f \in C^3({\overline{B}})\cap H(B)$,
then the functions
\begin{eqnarray}\label{def:discrete-int} f^h(\zeta):=\int_{\partial B^h}\mathcal{K}^h(z,\zeta)f(z)dS^h(z) \end{eqnarray}
are discrete holomorphic on $(B^h)^{\circ}$ and convergent to $f$ in the sense that
 $$\lim_{h\rightarrow0^+}\max_{B^h}|f-f^h|=0. $$
\end{theorem}
\begin{proof}
According to   Lemma \ref{T4}, we know that $B^h$ converges to $B$.
Since  $\mathcal{K}^h(z, \cdot)$ is discrete holomorphic on  $(B^h)^\circ$
 for any given $z\in \partial B^h$, it follows  that $f^h$ is also discrete holomorphic on $(B^h)^\circ$.

Now we prove that $f^h$ converges to $f$.
By  Theorem \ref{eq:BM} we have
$$f(\zeta)-f^h(\zeta)=\displaystyle{\int_{ B^h}} E^h(\zeta-z) \partial_{\bar z}^h f (z)dV^h(z) $$ for any  $\zeta \in B^h$ so that
$$|f(\zeta)-f^h(\zeta)|\leqq\max_{B^h}|\partial_{\bar z}^h f| \;\displaystyle{\int_{ B^h}}  \left|E^h(\zeta-z)\right| dV^h(z)$$
and by H\"older's  inequality
$$|f(\zeta)-f^h(\zeta)|\leqq\max_{B^h}|\partial_{\bar z}^h f| \left(\displaystyle{\int_{ B^h}}1dV^h(z)\right)^{2/3} \left(\displaystyle{\int_{ B^h}}|E^h(\zeta-z)|^3 dV^h(z)\right)^{1/3}. $$
Since the measure $V^h$ is invariant under group operations  of $\mathbb{Z}_h^2$, we have
\begin{eqnarray*}\int_{ B^h}|E^h(\zeta-z)|^3 dV^h(z)&\leqq& \int_{ Z_h^2}|E^h(z)|^3 dV^h(z)
\\
&=& \sum_{x, y\in \mathbb Z_h^2} \left|\frac{1}{h} E(\frac{x}{h}, \frac{y}{h})\right|^3h^2
\\
&=& \frac{1}{h}\sum_{s, t\in \mathbb Z^2} \left| E(s,t)\right|^3
\\
&=& \frac{1}{h}\int_{\mathbb Z^2} \left| E \right|^3dV
\\
&=& O(h^{-1}).
\end{eqnarray*}
The last step used Lemma \ref{T3}.
By  Lemmas \ref{T1} and  \ref{T2}, we have
 \begin{eqnarray*}\max_{B^h}|\partial_{\bar z}^h f|&=&O(h^2),
\\
\int_{ B^h}1dV^h&=&O(1).
\end{eqnarray*}
The above estimates thus yield
 $$\max_{B^h}|f-f^h|\leqq O(h^2)O(1)O(h^{-\frac{1}{3}})=O(h^{\frac{5}{3}}). $$
 This completes the proof.
\end{proof}

Notice that Theorem \ref{Conv} holds true merely on the specific sets $(B\bigcap\mathbb{Z}_h^2)^{\circ}$. If we impose slightly strong conditions
on $f$, we can show the convergence in more general cases.

\begin{theorem}\label{Conv2}
Let $B$ be a bounded open set in $\mathbb{C}$ and $f\in H(\overline{B})$. If $B^h$ converges to $B$, there exists $f^h: B^h\longrightarrow \mathbb{C}$  discrete holomorphic on $(B^h)^{\circ}$ and  convergent to $f$ in the sense that
$$\lim_{h\rightarrow0^+}\max_{B^h\bigcap B}|f-f^h|=0. $$
\end{theorem}
\begin{proof}

Take  an open set $U$ of $\mathbb C$ such that $f\in H(U)$ and $\overline B\subset U$.
If $B^h$ converges to $B$, we have $\partial B^h\subset U$ for any $h$ sufficiently  small. Therefore we can define
  $f^h$ as in (\ref{def:discrete-int}).
  With this modification,  the result follows from the same argument as in Theorem \ref{Conv} with  $B^h$ in place of $(B\bigcap\mathbb{Z}_h^2)^{\circ}$.
\end{proof}

\subsection{Convergence}

The scaling limit of  discrete holomorphic functions is shown to be  holomorphic in this subsection.

\begin{theorem}\label{INV-Conv}
Let $B$ be a bounded open set in $\mathbb{C}$
and  $B^h\subset\mathbb Z_h^2$  convergent to $B$. If $f^h: B^h\longrightarrow\mathbb C$ is discrete
holomorphic on $(B^h)^\circ$   and  convergent to a function   $f\in C({B})$ in the sense that
$$\lim_{h\rightarrow0^+}\max_{B^h\bigcap B}|f-f^h|=0,$$
then $f\in H(B)$.
\end{theorem}

Its proof relies on   the following  key  fact.

\begin{proposition}\label{Interior-cover}
Let $B$ be a bounded open set in $\mathbb{C}$
and  $B^h\subset\mathbb Z_h^2$  convergent to $B$. Then for any $U\subset\subset B$, there exists $\delta>0$ such that when $h<\delta$, we have $$U\bigcap \mathbb{Z}_h^2 \subset B^h.$$
\end{proposition}
\begin{proof} Without loss of generality, we can
assume that $U=B(z_0, R)\subset\subset B$ due to compactness.

Since $B^h$ is convergent to $B$, by definition we  have
$$\lim_{h\rightarrow 0^+}\max_{\alpha \in \overline B}\min_{\beta \in B^h}||\alpha-\beta||=0, $$
which means that for any given $\epsilon>0$
$$B^h+B(0,\epsilon)\supset \overline B\ni z_0,$$
provided $h$  sufficiently small.
In particular, taking $\epsilon=R$, we have
$$ B^h+B(0,R)\ni z_0,$$
i.e.,
$(z_0-B(0,R))\cap  B^h\neq \phi.$
By assumption $U=B(z_0, R)=z_0-B(0, R)$, this means
$$U\bigcap B^h\neq \emptyset$$
 when $h$ is small enough.

On the other hand, we denote
$$d={\rm{dist}}(U, \partial B):=\inf\limits_{\alpha\in U, \beta \in \partial B}||\alpha-\beta||.$$
Again the convergence of $B^h$  implies
 $$\lim_{h\rightarrow 0^+}\max_{\alpha \in \partial B^h}\min_{\beta \in \partial B}||\alpha-\beta||=0. $$  For any $h$ sufficiently small, we then have
$$\partial B^h\subset \partial B + B(0,d)$$
 so that
$$\partial B^h\bigcap U \subset  (\partial B + B(0,d))\bigcap U=\emptyset$$
since $d={\rm{dist}}(U, \partial B)$.

Now we have proved that   $$U\bigcap  B^h \neq\emptyset, \qquad U\bigcap \partial B^h =\emptyset$$  for any  $h$ small enough. Based on these facts, we come to show that
$$U\bigcap \mathbb{Z}_h^2 \subset B^h, \qquad \forall\  h<<1.$$

Assume this is not valid, then there exists
$h^*>0$
such that
$$U\bigcap B^{h^*}\neq \emptyset, \qquad U\bigcap\partial B^{h^*} =\emptyset,\qquad
U\bigcap\mathbb Z_{h^*}^2\not\subset B^{h^*}.$$
Now we take two elements
$$\alpha\in U\bigcap B^{h^*}=U\bigcap Z_{h^*}^2\bigcap B^{h^*}, \qquad \beta\in (U\bigcap \mathbb Z_{h^*}^2) \setminus B^{h^*}.$$
Since
$$\alpha, \beta\in U\bigcap\mathbb{Z}_{h^*}^2=B(z_0, R)\bigcap\mathbb{Z}_{h^*}^2$$
and the last set is   discrete connected in the square lattice $\mathbb{Z}_{h^*}^2$, there exist $\{z_k\}_{k=1}^{m} \subset U\bigcap\mathbb{Z}_{h^*}^2$ such that
$$z_1=\alpha, \qquad z_m=\beta, \qquad z_{k+1}\in N(z_{k})$$   for any $k=1,2,\dots,m-1$. Notice that
$\alpha\in  B^{h^*}$ and $\beta\notin B^{h^*}$,
we can take $1\leq k^*\leq m-1$  such that
$$z_{k^*}\in  B^{h^*}, \qquad z_{k^*+1}\notin B^{h^*}.$$
They are both in $\partial B^{h^*}$ since $z_{k^*+1}\in N(z_{k^*})$. This implies that
$$z_{k^*}\in \partial  B^{h^*}\bigcap U,$$ which violates the assumption $\partial B^{h^*}\bigcap U =\emptyset$. This completes the proof.
\end{proof}

The preceding proposition   results in  the  $w^*$-convergence of discrete measures, which is essential in the proof of Theorem \ref{INV-Conv}.
\begin{lemma}\label{W*Conv}
Let $B$ be a bounded open set in $\mathbb{C}$
and  $B^h\subset\mathbb Z_h^2$  convergent to $B$. Then we have
$$\lim_{h\rightarrow 0^+}\int_{B^h\bigcap B}fdV^h=\int_{B}fdV $$ for any $f\in C_c(B)$.
That is,
$$w^*-{\lim}_{h\to 0^+} V^h\left.\right|_{B^h\cap B}= V\left.\right|_{B} \quad \text{ in }\quad \big(C_c(B)\big)^*. $$

\end{lemma}
\begin{proof}
For any given $f\in C_c(B)$, we can take $U$ to be a finite union of balls
 such that  \begin{eqnarray}\label{def:subset-B-h}
 {\bf{supp}}f \subset U \subset\subset B.
\end{eqnarray}

In view of Proposition \ref{Interior-cover},  for $h$ sufficiently small we have
$$U\bigcap \mathbb{Z}_h^2 \subset B^h \subset \mathbb Z_h^2,$$
which intersects with $U$  to yield  $B^h\bigcap U=U\bigcap \mathbb{Z}_h^2$.
This together with (\ref{def:subset-B-h})  concludes that
 $$\int_{B^h\bigcap B}fdV^h=\int_{B^h\bigcap U}fdV^h=\int_{U\bigcap \mathbb{Z}_h^2}fdV^h. $$
The last integral  is identical to a Riemann sum  of a certain Riemann integral, which means $$\lim_{h\rightarrow 0^+}\int_{B^h\bigcap B}fdV^h=\lim_{h\rightarrow 0^+}\int_{U\bigcap \mathbb{Z}_h^2 }fdV^h=\int_{U}fdV=\int_{B}fdV. $$
\end{proof}

Now we can give the proof of our main result in this subsection.

\begin{pot2}
Let  $f^h: B^h\longrightarrow\mathbb C$ be discrete
holomorphic on $(B^h)^\circ$   and  convergent to a function   $f\in C({B})$.
In order to  prove   $f\in H(B)$,  we only need   to verify
$$\int_{B}f\partial_{\bar z}\phi dV=0 $$
for any  $\phi\in C_0^{\infty}(B)$, i,e, $f$ is  holomorphic in the sense of distribution.
By Lemma \ref{W*Conv}, $$\lim_{h\rightarrow 0^+}\int_{B^h\bigcap B}f\partial_{\bar z}\phi dV^h=\int_{B}f\partial_{\bar z}\phi dV. $$
It remains  to prove $$\lim_{h\rightarrow 0^+}\int_{B^h\bigcap B}f\partial_{\bar z}\phi dV^h=0. $$

 We now  separate the last integral into three parts \begin{eqnarray}\label{def:I-h-integral}\displaystyle{\int_{B^h\bigcap B}f\partial_{\bar z}\phi dV^h}=I_1^h+I_2^h+I_3^h,\end{eqnarray}
 where
\begin{eqnarray*}
I_1^h&=&\displaystyle{\int_{B^h\bigcap B}f^h\partial_{\bar z}^h\phi dV^h}, \\
I_2^h&=&\displaystyle{\int_{B^h\bigcap B}(f-f^h)\partial_{\bar z}^h\phi dV^h}, \\
I_3^h&=&\displaystyle{\int_{B^h\bigcap B}f(\partial_{\bar z}\phi- \partial_{\bar z}^h\phi)dV^h}.
\end{eqnarray*}

For  the first term $I_1^h$, since $\phi\in C_0^{\infty}(B)$ we have  $${\bf supp}\partial_{\bar z}^h\phi \subset B$$
for $h$ sufficiently  small. Hence
  $$I_1^h=\displaystyle{\int_{B^h}f^h\partial_{\bar z}^h\phi dV^h}$$
and by  Theorem \ref{Stoke} we get
\begin{equation}\label{Eq-INV-Conv}\begin{array}{lll}
I_1^h                                               &=&\displaystyle{\frac{1}{4}\int_{\partial B^h}\Big(\phi(z+h)n_1^{+}(z)+\phi(z-h)n_1^{-}(z)+i\phi(z+ih)n_2^{+}(z)}\\
                                                & &+i\phi(z-ih)n_2^{-}(z) \Big)f^h(z)dS^h(z)-\displaystyle{\int_{B^h}\phi\partial_{\bar z}^hf^hdV^h}.\end{array}\end{equation}

Pick an open set $U$ such that
$${\bf supp}\phi\subset \subset U\subset\subset B$$
and let $h$ be small enough obeying
$${\bf supp} \phi+B(0,4h) \subset U.$$
By Proposition \ref{Interior-cover}, we have
 $$U\bigcap \mathbb{Z}_h^2 \subset B^h, \qquad\forall\  h<<1.$$  As a result,$$\big({\bf supp} \phi+B(0,4h)\big)\bigcap \mathbb{Z}_h^2 \subset B^h.$$
This implies  if $$\alpha \in \big({\bf supp}\phi+B(0,2h)\big)\bigcap \mathbb{Z}_h^2,$$ then
 $$N(\alpha)\subset B(\alpha,2h)\bigcap\mathbb{Z}_h^2 \subset B^h$$ so that $$\big({\bf supp} \phi+B(0,2h)\big)\bigcap \mathbb{Z}_h^2 \subset (B^h)^{\circ}.$$
Consequently,
$$\big({\bf supp} \phi+B(0,2h)\big)\bigcap \partial B^h=\emptyset $$
which means   that $\phi$, $\phi(\cdot\pm h)$ and $\phi(\cdot \pm ih)$ all vanish on $\partial B^h$.
   It follows that $$\displaystyle{\int_{\partial B^h}\big(\phi(\cdot+h)n_1^{+}+\phi(\cdot-h)n_1^{-}+i\phi(\cdot+ih)n_2^{+}}+i\phi(\cdot-ih)n_2^{-} \big)f^hdS^h=0 $$ and
$$\displaystyle{\int_{B^h}\phi\partial_{\bar z}^hf^hdV^h}=\displaystyle{\int_{B^h\setminus \partial B^h}\phi\partial_{\bar z}^hf^hdV^h}=\displaystyle{\int_{(B^h)^{\circ}}\phi\partial_{\bar z}^hf^hdV^h}=0. $$
The last step used the fact that $f^h$ is discrete holomorphic.
Hence  $I_1^h$ vanishes for sufficiently small $h$ according to  (\ref{Eq-INV-Conv}).

Next we estimate  the second item $I_2^h$ in (\ref{def:I-h-integral}). By definition,
\begin{equation*}
\begin{array}{lll}
|I_2^h|&\leqq&\displaystyle{\int_{B^h\bigcap B}|(f-f^h)\partial_{\bar z}^h\phi| dV^h}\\
     &\leqq&\max\limits_{B^h\bigcap B}|f-f^h|\max\limits_{\mathbb{R}^2}|\partial_{\bar z}^h\phi|\displaystyle{\int_{B^h\bigcap B}}1dV^h.
\end{array}
\end{equation*}
Since $B$ is bounded, we can assume $B\subset B(0, R)$ for some $R>0$ so that
$$B^h\bigcap B\subset B(0, R)\cap \mathbb Z_h^2.$$
As  shown in the proof of Lemma \ref{W*Conv},
$$\lim_{h\to 0^+} \int_{B(0, R)\cap \mathbb Z_h^2}1dV^h= \int_{B(0, R)}1dV$$
and this means
$$\int_{B^h\bigcap B}1dV^h=O(1). $$
Since $\phi\in C_0^{\infty}(B)$, we have $$\max\limits_{\mathbb{R}^2}|\partial_{\bar z}^h\phi|=O(1). $$
In addition, according to assumption, $$\max\limits_{B^h\bigcap B}|f-f^h|=o(1). $$
The above facts together imply that  $I_2^h=o(1). $

Now it remains  to show that $$I_3^h=o(1).$$
By definition
\begin{equation*}
\begin{array}{lll}
|I_3^h|&\leqq&\displaystyle{\int_{B^h\bigcap B}|f(\partial_{\bar z}\phi-\partial_{\bar z}^h\phi)| dV^h}\\
     &\leqq&\max\limits_{{\bf supp}\phi+\overline{B(0,h)}}|f|\max\limits_{\mathbb{R}^2}|\partial_{\bar z}\phi-\partial_{\bar z}^h\phi|\displaystyle{\int_{B^h\bigcap B}}1dV^h.
\end{array}
\end{equation*}
Since $f \in C(B)$ and $\phi \in C_0^{\infty}(B)$, we have
$$\max\limits_{{\bf supp}\phi+\overline{B(0,h)}}|f|=O(1), \qquad
\max\limits_{\mathbb{R}^2}|\partial_{\bar z}\phi-\partial_{\bar z}^h\phi|=o(1). $$
As we have observed, $$\displaystyle{\int_{B^h\bigcap B}}1dV^h=O(1), $$
we thus obtain $I_3^h=o(1)$ and this completes the proof.
\end{pot2}

\begin{remark} The  convergence problem of discrete holomorphic fermions has been considered in the study of the invariance of Ising model  by Smirnov \cite{38}.
 His approach relies  on an important fact that  the discrete holomorphic fermion is  a solution of the discrete Riemann boundary value problem.
\end{remark}

\section{Uniform convergence of derivatives}\label{S7}
 The uniform convergence of derivatives up to second order
is shown  for the family of
the discrete holomorphic functions  in this section.

Notice that if  $f^h$ is defined on $B^h$, then its first order derivatives    can only be defined in
 $(B^h)^{\circ}$ and  similarly for  its second  order  derivatives.

If $B^h$ converges to $B$,
 we find
  $B^h \bigcap B$ is almost the same as   $B\bigcap \mathbb Z_h^2$ in view of Proposition \ref{Interior-cover} and, by direct verification,
$$\lim_{h\to {0^+}} (B^h)^{\circ}=B,\qquad \lim_{h\to {0^+}} (B^h)^{\circ \circ}=B.$$
Here    $(B^h)^{\circ \circ}$ stands for the discrete interior of $(B^h)^{\circ}$.

\begin{theorem}\label{Conv-D}
Let $B$ be a bounded open set and $f\in H(\overline{B})$. If $B^h$ converges to $B$, then  there exists $f^h: B^h\longrightarrow \mathbb{C}$   discrete holomorphic on $(B^h)^{\circ}$ and   uniformly convergent to $f$ in the sense that
\begin{eqnarray*}&& \lim_{h\rightarrow0^+}\max_{B^h\bigcap B}|f-f^h|=0,
\\
&& \lim_{h\rightarrow0^+}\max_{(B^h)^{\circ}\bigcap B}|\partial_{z}f-\partial_{z}^hf^h|=0,
\\
&& \lim_{h\rightarrow0^+}\max_{(B^h)^{\circ \circ}\bigcap B}|(\partial_{z})^2f-(\partial_{z}^h)^2f^h|=0.
\end{eqnarray*}

\end{theorem}
\begin{proof}
Take  an open subset  $U$  of $\mathbb C$ such that $f\in H(U)$ and $U\supset \bar B$.
When $h$ is sufficiently small, we have  $\partial B^h\subset U$ since $B^h$ converges to $B$. Therefore we can define
  $f^h$ as in (\ref{def:discrete-int}), i.e.,
$$f^h(\zeta)=\int_{\partial B^h}\mathcal{K}^h(z,\zeta)f(z)dS^h(z). $$
The first identity in Theorem \ref{Conv-D} has been  shown in Theorem \ref{Conv2}.

To prove the last two identities, notice that
$$\lim_{h\rightarrow0^+}\max_{ B}\Big\{|\partial_{z}f-\partial_{z}^hf|+|(\partial_{z})^2f-(\partial_{z}^h)^2f|\Big\}=0 $$
since $\partial^h$ converges to $\partial$ in $C^2(B)$.
Therefore, it is  sufficient  to show that
$$\lim_{h\rightarrow0^+}\max_{(B^h)^{\circ}\bigcap B}|\partial_{z}^hf-\partial_{z}^hf^h|=0, $$
$$\lim_{h\rightarrow0^+}\max_{(B^h)^{\circ \circ}\bigcap B}|(\partial_{z}^h)^2f-(\partial_{z}^h)^2f^h|=0. $$

According to Theorem \ref{MCIEq11}, we have $$f=f^h+\int_{B^h}E^h(\cdot-z)\partial_{\bar z}^hf(z)dV^h(z)\qquad \text{ on }B^h.$$  Applying   difference operators on both sides, we obtain
$$\partial_{z}^hf-\partial_{z}^hf^h=\int_{B^h}\partial_{z}^hE^h(\cdot-z)\partial_{\bar z}^hf(z)dV^h(z)\qquad \text{ on }(B^h)^{\circ}$$
and
$$(\partial_{z}^h)^2f-(\partial_{z}^h)^2f^h=\int_{B^h}(\partial_{z}^h)^2E^h(\cdot-z)\partial_{\bar z}^hf(z)dV^h(z) \qquad \text{ on }(B^h)^{\circ \circ}.$$
These together with the H\"older  inequality thus imply  that
$$\max_{(B^h)^{\circ}\bigcap B}|\partial_{z}^hf-\partial_{z}^hf^h|\leqslant\max_{B^h}|\partial_{\bar z}f|\left(\int_{\mathbb{Z}_h^2}
(\partial_{z}^hE^h)^2dV^h\right)^{1/2}\left(\int_{B^h}1dV^h\right)^{1/2} $$
and
$$\max_{(B^h)^{\circ\circ}\bigcap B}|(\partial_{z}^h)^2f-(\partial_{z}^h)^2f^h|\leqslant\max_{B^h}|\partial_{\bar z}f|\int_{\mathbb{Z}_h^2}|(\partial_{z}^h)^2E^h|dV^h. $$

On the other hand,  according to Lemmas \ref{T1} and \ref{T2}, we know
$$\max\limits_{B^h}|\partial_{\bar z}f|=O(h^2), \qquad \displaystyle{\int_{B^h}1dV^h=O(1)}.$$ By Lemma \ref{Int-FS}, we obtain
$$\int_{\mathbb{Z}_h^2}(\partial_{z}^hE^h)^2dV^h=
h^2\int_{\mathbb{Z}^2}\left(\frac{1}{h^2}\partial_{z}^1E^1\right)^2dV^1
=\frac{1}{h^2}||\partial_{z}^1E^1||_{L^2}^2=O(\frac{1}{h^2})$$
and
$$\int_{\mathbb{Z}_h^2}|(\partial_{z}^h)^2E^h|dV^h=h^2\int_{\mathbb{Z}^2}|
\frac{1}{h^3}(\partial_{z}^1)^2E^1|dV^1=
\frac{1}{h}||(\partial_{z}^1)^2E^1||_{L^1}=O(\frac{1}{h}). $$
Altogether,  the above results yield
 $$\max_{(B^h)^{\circ}\bigcap B}|\partial_{z}^hf-\partial_{z}^hf^h|\leqslant O(h^2)O(\frac{1}{h})O(1)=O(h), $$
and
$$\max_{(B^h)^{\circ\circ}\bigcap B}|(\partial_{z}^h)^2f-(\partial_{z}^h)^2f^h|\leqslant O(h^2)O(\frac{1}{h})=O(h). $$
This completes the proof.
\end{proof}

\begin{remark} Theorem \ref{Conv-D}  holds true under the weak condition that $f\in C^2(\overline{B})\bigcap H(B)$ provided that the specific case that $B^h=(B\bigcap\mathbb{Z}_h^2)^{\circ}$ is considered.
\end{remark}

\section{Appendix:  Technical lemmas}\label{S9}

Some technical results are included in this appendix.

\begin{remark} For any $B \subset \mathbb{Z}_h^2$,  its discrete boundary has two layers:
\begin{eqnarray*}
\partial^{+} B&:=&\partial B \cap B,
\\
 \partial^{-} B&:=&\partial B \setminus B.
\end{eqnarray*}
Either both are empty or both not.
\end{remark}
\begin{lemma}\label{T4} Let $B$ be a bounded open set in $\mathbb{C}$ and $B^h=(B\bigcap\mathbb{Z}_h^2)^{\circ}$. Then
$B^h$  converges to $B$.
\end{lemma}
\begin{proof} To prove that $\lim\limits_{h\rightarrow 0^{+}}B^h=B$, we need to verify four conditions in Definition \ref{CDC}.

(i) Since $B^h$ is contained in $B$, we have
$$\min_{\beta\in \overline{B}}||\alpha-\beta||=0, \qquad \forall\ \alpha\in B^h $$
so that
$$\lim_{h\rightarrow 0^+}\max_{\alpha \in B^h}\min_{\beta\in \overline{B}}||\alpha-\beta||=0. $$

(ii) Next we want to show  $$\lim_{h\rightarrow 0^+}\max_{\alpha \in \overline{B}}\min_{\beta\in B^h}||\alpha-\beta||=0, $$
i.e., for any $\epsilon>0$, there exists $\delta>0$, such that when $h\leq \delta$
we have
\begin{eqnarray}\label{def:lim-2-bh}\overline B\subset B^h+B(0, 2\epsilon).
\end{eqnarray}

To prove (\ref{def:lim-2-bh}), we let  $\epsilon >0$ and $z \in \overline{B}$ be given. Since $B$ is open, we can pick  $\delta>0$ and $z'\in B$ such that
$$B(z',\delta) \subset\subset B \bigcap B(z,\epsilon).$$
Therefore,
$$ B(z',\delta)\bigcap\mathbb{Z}_h^2\subset (B\bigcap\mathbb{Z}_h^2)^{\circ}= B^h$$
for any  $h$  sufficiently small. This implies
 $$B(z',\delta) \bigcap B^h = B(z',\delta) \bigcap \mathbb Z_h^2\neq \emptyset, \qquad \forall\ h<<1.$$
Since $B(z,\epsilon)\bigcap B^h\supset  B(z',\delta) \bigcap B^h$, it follows that
\begin{eqnarray} \label{Id:nonempity} B(z,\epsilon)\bigcap B^h\neq \emptyset, \qquad \forall\ h<<1.
 \end{eqnarray}

 Since $\overline{B}$ is compact, one can find $z_1,\dots,z_n$ such that $$\overline{B}\subset \bigcup\limits_{k=1}^{n}B(z_k,\epsilon).$$  As we have derived there exists $\eta_k>0$ such that $$B(z_k,\epsilon) \bigcap B^h \neq \emptyset$$ whenever $h<\eta_k$. Hence when $h<\min\limits_{1\leq k\leq n}\eta_k$, we have $$\overline{B}\subset B^h+B(0,2\epsilon) $$
as desired.

(iii) Now we  come to show
$$\lim_{h\rightarrow 0^+}\max_{\alpha \in\partial B^h}\min_{\beta\in\partial B}||\alpha-\beta||=0.$$
It is sufficient to verify that $$\min_{\beta\in\partial B}||z-\beta||\leqq 2h, \quad\forall\ z \in \partial B^h. $$
Its proof is split into two cases:

(a) case 1: $z \in \partial^{+}B^h$.

For any $\zeta \in \partial^{+}(B\bigcap\mathbb{Z}_h^2)$,
by definition we have $\zeta \in B$ and there is a point $\zeta'\in\mathbb{Z}_h^2\setminus B$ with $||\zeta'-\zeta||=h$. Since $\zeta \in B$ and $\zeta'\notin B$, we obtain $$\min_{\beta\in\partial B}||\zeta-\beta||\leqq h, \quad\forall\ \zeta \in \partial^{+}(B\bigcap\mathbb{Z}_h^2). $$

For any $z \in \partial^{+}B^h$, there exists $\zeta \in N(z)$ such that $\zeta\notin B^h$. Since $B^h$ is the discrete interior of $B\bigcap\mathbb{Z}_h^2$, we have $\zeta \in \partial^{+}(B\bigcap\mathbb{Z}_h^2)$, which implies $$\min_{\beta\in\partial B}||z-\beta||\leqq ||z-\zeta||+\min_{\beta\in\partial B}||\zeta-\beta||\leqq 2h. $$

(b) case 2: $z \in \partial^{-}B^h$.

By assumption we have  $B^h=(B\bigcap\mathbb{Z}_h^2)^{\circ}$ so that $$z \in \partial^{+}(B\bigcap\mathbb{Z}_h^2).$$ As shown above in case 1, we deduce that
$$\min_{\beta\in\partial B}||z-\beta||\leqq h. $$

(iv) Finally, we come to prove $$\lim_{h\rightarrow 0^+}\max_{\alpha \in\partial B}\min_{\beta\in\partial B^h}||\alpha-\beta||=0. $$
Since $\partial B$ is compact, it is sufficient to show that  for any $\epsilon >0$,
 there exists $\delta>0$ such that for any $h<\delta$ and $z\in \partial B$ we have $$\min_{\beta \in \partial B^h}||z-\beta||<\epsilon, $$

To verify this fact, we let $\epsilon >0$ and $z\in \partial B$ be given.
As shown in  (\ref{Id:nonempity}),  there exists $\eta>0$ such that  when $h<\eta$ we have $$B(z,\epsilon)\bigcap B^h \neq \emptyset.$$

 On the other hand, since $B(z,\epsilon)\setminus \overline{B}$ is open, there exist $\eta'>0$ such that $$B(z,\epsilon)\bigcap (\mathbb{Z}_h^2 \setminus \overline{B}) \neq \emptyset, \qquad \forall\ h<\eta'.$$ Thus if $h<\min\{\eta,\eta'\}$, one can find a sequence $$\{z_k\}_{k=1}^n \in \mathbb{Z}_h^2\bigcap B(z,\epsilon)$$
 with  the following properties:
$$\begin{array}{c}
z_1 \in B^h, \quad z_n \in \mathbb{Z}_h^2 \setminus B^h,  \quad
||z_k-z_{k+1}||=h \quad (1\leq  k\leq n-1).
\end{array}
$$
Now we take $$k_0=\max\{k: z_k \in B^h\},$$ which  means
$$z_{k_0} \in B^h, \qquad z_{k_0+1} \notin B^h.$$  According to the definition of discrete boundary, we have $z_{k_0} \in \partial B^h$, which implies that
$$\min_{\beta \in \partial B^h}||z-\beta||<\epsilon, \quad\forall \  h<\min\{\eta,\eta'\}. $$
This completes the proof.
\end{proof}

\begin{lemma}\label{T1}
Let $B$ be a bounded open set in $\mathbb{C}$ and $f \in C^3({\overline{B}})\cap H(B)$. If $B^h\subset B$, then we have
$$\max_{B^h}|\partial_{\bar z}^hf|=O(h^2). $$
\end{lemma}

\begin{proof}Since $f$ is holomorphic, we have
$$\partial_{\bar z}^hf(z)=\frac{f(z+h)-f(z-h)-\partial_xf(z)2h}{2h}+i\frac{f(z+ih)-f(z-ih)-\partial_yf(z)2h}{2h}. $$ The conclusion follows  obviously from the mean value theorem.
\end{proof}

\begin{lemma}\label{T2}
Let $B$ be a bounded open set in $\mathbb{C}$. If $B^h$ converges to $B$,
then we have
$$\displaystyle{\int_{ B^h}}1dV^h=O(1). $$
\end{lemma}
\begin{proof}
Since $B$ is bounded we may assume $B \subset\subset B(0,R)$. Therefore, when $h$ is small enough, we have $B^h\subset B(0,R)$ so that
 $$\displaystyle{\int_{ B^h}}1dV^h\leqq\displaystyle{\int_{ \mathbb{Z}_h^2}}\chi_{B(0,R)}dV^h.$$

According to the definition of the Haar measure $V^h$, the last integral  is identical to a Riemann sum, so that
$$\lim_{h\rightarrow0^+}\displaystyle{\int_{ \mathbb{Z}_h^2}}\chi_{B(0,R)}dV^h=\pi R^2. $$
Therefore,
$$\displaystyle{\int_{ B^h}}1dV^h =O(1) $$
as desired.
\end{proof}

\begin{lemma}\label{T3}
Function $E$ in (\ref{Fund-Sol}) belongs to $L^3(\mathbb{Z}^2)$.
\end{lemma}
\begin{proof}
Recall that
$$E(x,y)=\frac{1}{4\pi^2}\int_{[-\pi,\pi]^2}\frac{2}{i\sin{u}-\sin{v}}e^{i(ux+vy)}dudv. $$
Namely, $E$ is the  Fourier   transform of function ${1}/({i\sin{u}-\sin{v}})$ up to a constant factor. The  Hausdorff-Young inequality thus tells us  that it is sufficient to show
 $$\frac{1}{i\sin{u}-\sin{v}} \in L^{{3}/{2}}(\mathbb{T}^2).$$
Its singular points are given by the solutions of equations
 $$\sin u=\sin v=0,$$
 or rather
 $$(u, v)\in\{0, \pi, -\pi\}^2.$$
 By periodicity, we only need to consider the singular point $(0, 0)$.
 At this point we have
 $$\frac{1}{|i\sin{u}-\sin{v}|}\simeq \frac{1}{\sqrt {u^2+v^2}}.$$
 It is evident that
 $$\frac{1}{\sqrt {u^2+v^2}} \in L^{{3}/{2}}(\mathbb{T}^2),$$
since
 $$\int_{[-\pi, \pi]^2} \left(\frac{1}{\sqrt {u^2+v^2}}\right)^{3/2}dudv<\infty.$$
This completes the proof. \end{proof}

 Next we show the integrability of the discrete derivatives of $E$ up to second order.
\begin{lemma}\label{Int-FS}
 The discrete derivatives of $E$  obey the integrability:

$$\partial_{z}^1E\in L^2(\mathbb{Z}^2), \qquad (\partial_{z}^1)^2E\in L^1(\mathbb{Z}^2). $$
\end{lemma}
\begin{proof}
By definition, $$E(x,y)=\frac{1}{4\pi^2}\int_{[-\pi,\pi]^2}\frac{2}{i\sin{u}-\sin{v}}e^{i(ux+vy)}dudv. $$
A direct calculation shows that the discrete derivatives of $E$ has integral representations
\begin{eqnarray}\label{eq:order-one-derivative}\partial_{z}^1E(x,y)=\frac{1}{4\pi^2}\int_{[-\pi,\pi]^2}\frac{i\sin{u}
+\sin{v}}{i\sin{u}-\sin{v}}e^{i(ux+vy)}dudv \end{eqnarray}
and
\begin{eqnarray}\label{eq:order-two-derivative}(\partial_{z}^1)^2E(x,y)=\frac{1}{8\pi^2}\int_{[-\pi,\pi]^2}\frac{(i\sin{u}+\sin{v})^2}{i\sin{u}-\sin{v}}e^{i(ux+vy)}dudv. \end{eqnarray}

From (\ref{eq:order-one-derivative}), we find that up to a constant  $\partial_{z}^1E(x,y)$ is the Fourier transform of the function $$\displaystyle{\frac{i\sin{u}+\sin{v}}{i\sin{u}-\sin{v}} }\in L^2(\mathbb{T}^2).$$   The Hausdorff-Young inequality yields
 $$\partial_{z}^1E \in L^2(\mathbb{Z}^2).$$

Similarly,  (\ref{eq:order-two-derivative}) shows that up to a constant  $\partial_{z}^2E(x,y)$ is the Fourier transform of the function
$$\displaystyle{\frac{(i\sin{u}+\sin{v})^2}{i\sin{u}-\sin{v}}}.$$  It is easy to see that the last function, together with its derivatives up to second order,   belongs to $L^{3/2}(\mathbb{T}^2)$. Again by the Hausdorff-Young inequality, this implies that $$(1+x^2+y^2)(\partial_{z}^1)^2E\in L^3(\mathbb{Z}^2).$$  Since $(1+x^2+y^2)^{-1}\in L^{3/2}(\mathbb{Z}^2)$,  the H\"older inequality shows that $(\partial_{z}^1)^2E\in L^1(\mathbb{Z}^2)$. This completes the proof.
\end{proof}

Theorem \ref{MCIEq88} is about the holomorphicity of the Bochner-Matinelli kernel. We finally give its proof.

\begin{pot4}
By definition, we have
 \begin{eqnarray*}\partial_{\bar\zeta}^h\mathcal{K}^h(z,\zeta)
&=&\partial_{\bar\zeta}^hA(z-\zeta)n_1^{-}(z)+\partial_{\bar\zeta}^hB(z-\zeta)n_1^{+}(z)\\
&&+i\partial_{\bar\zeta}^hC(z-\zeta)n_2^{-}(z)+i\partial_{\bar\zeta}^hC(z-\zeta)n_2^{+}(z).                                                                 \end{eqnarray*}
We now calculate each summand on the left. By direct calculation, we get
 \begin{eqnarray*}                                                                    \partial_{\bar\zeta}^hA(z-\zeta)n_1^{-}(z)
 &=&-4^{-1}\partial_{\bar\zeta}E(h-z+\zeta)n_1^{-}(z)\\
 &=&-4^{-1}\delta_0^h(h-z+\zeta)n_1^{-}(z).
             \end{eqnarray*}
Similarly,
 \begin{eqnarray*}                                                                     \partial_{\bar\zeta}^hB(z-\zeta)n_1^{+}(z)
 &=&-4^{-1}\partial_{\bar\zeta}^hE(-h-z+\zeta)n_1^{+}(z)\\
 &=&-4^{-1}\delta_0^h(-h-z+\zeta)n_1^{+}(z)
             \end{eqnarray*}

 \begin{eqnarray*}                                                                     \partial_{\bar\zeta}^hC(z-\zeta)n_2^{-}(z)
 &=&-4^{-1}\partial_{\bar\zeta}^hE(hi-z+\zeta)n_2^{-}(z)\\
 &=&-4^{-1}\delta_0^h(hi-z+\zeta)n_2^{-}(z)
             \end{eqnarray*}
and
 \begin{eqnarray*}                                                                     \partial_{\bar\zeta}^hD(z-\zeta)n_2^{+}(z)
 &=&-4^{-1}\partial_{\bar\zeta}^hE(-hi-z+\zeta)n_2^{+}(z)\\
 &=&-4^{-1}\delta_0^h(-hi-z+\zeta)n_2^{+}(z).
             \end{eqnarray*}
Combining the above results to yield
\begin{eqnarray*}-4\partial_{\bar\zeta}^h\mathcal{K}^h(z,\zeta)
&=&\delta_0^h(h-z+\zeta)n_1^{-}(z)+\delta_0^h(-h-z+\zeta)n_1^{+}(z)\\
&&+i\delta_0^h(h-z+\zeta)n_2^{-}(z)
+ i\delta_0^h(-h-z+\zeta)n_2^{+}(z).
             \end{eqnarray*}
It follows immediately that $$\partial_{\bar\zeta}^h\mathcal{K}^h(z,\zeta)=0 \qquad\forall \quad\zeta \not\in N(z).$$
since $\delta_0^h(h-z+\zeta)$, $\delta_0^h(ih-z+\zeta)$ and $\delta_0^h(-ih-z+\zeta)$ all vanish whenever $\zeta \not\in N(z)$.

We now come to show that
$$\partial_{\bar\zeta}^h\mathcal{K}^h(z,\zeta)=0 \text{ for any }\zeta \not\in \partial B. $$
We only need to show the first item in the right side vanishes, i.e.
$$\delta_0^h(h-z+\zeta)n_1^{-}(z)=0, \qquad \forall \zeta\in \mathbb Z_h^2\setminus \partial B,\quad \forall z\in\partial B,$$
 since the remaining items also vanishes similarly.

 We may assume  $n_1^{-}(z)\neq 0$, otherwise there is nothing to prove.
 In this case, we have $$\partial_1^{-,h}\chi_B(z)\neq 0.$$
  That is,
 $$\chi_B(z)-\chi_B(z-h)\neq 0.$$
  This implies that either
  $z\in B $, $z-h \not\in B$ or $z \not\in B$, $z-h \in B$.
  In both cases, we have both  $z$ and  $z-h$ are in the boundary  $\partial B$
  so that
  $\delta_0^h(h-z+\zeta)=0$ for any $\zeta \not\in \partial B$.
  Hence the first item vanishes.
 \end{pot4}

\vskip2cm

\bibliographystyle{amsplain}

\end{document}